\documentclass[letterpaper,12pt]{amsart}

\usepackage{amssymb}
\usepackage{amsmath,amscd}
\usepackage{rotating}

\numberwithin{equation}{section}

\newtheorem{lemma}{Lemma}[section]

\newtheorem{theorem}[lemma]{Theorem}
\newtheorem{proposition}[lemma]{Proposition}
\newtheorem{definition}[lemma]{Definition}
\newtheorem{corollary}[lemma]{Corollary}

\newtheorem{notation}[lemma]{Notation}

\theoremstyle{definition}
\newtheorem{remark}[lemma]{Remark}

\newcommand{\Gal}{\mathrm{Gal}}
\newcommand{\Hol}{\mathrm{Hol}}

\newcommand{\Aut}{\mathrm{Aut}}

\newcommand{\Aff}{\mathrm{Aff}}

\newcommand{\GL}{\mathrm{GL}}
\newcommand{\F}{\mathbb{F}}
\newcommand{\Z}{\mathbb{Z}}
\newcommand{\NN}{\mathcal{N}}

\newcommand{\ord}{\mathrm{ord}}
\newcommand{\id}{\mathrm{id}}

\newcommand{\conj}{\mathrm{conj}}
\newcommand{\opp}{\mathrm{op}}

\newcommand{\tkap}{\widetilde{\kappa}}

\newcommand{\Mat}{\mathrm{Mat}}
\newcommand{\Autsb}{\mathrm{Aut}_{\mathrm{sb}}}
	
\begin{document}
\title{On a family of simple skew braces}

\author{Nigel P.~Byott}
\address{Department of Mathematics \& Statistics, University of Exeter, Exeter 
EX4 4QF, U.K.}  
\email{N.P.Byott@exeter.ac.uk}

\thanks{This work was supported by the Engineering and Physical Sciences Research Council [grant number EP/V005995/1]. \newline 
	\indent The author thanks Andrea Caranti, Andrew Darlington, Henri Johnston and Paul Truman for comments on earlier versions of this work. \newline
	\indent
	For the purpose of open access, the author has applied a `Creative Commons Attribution (CC BY) licence  to any Author Accepted Manuscript version arising from this submission'. \newline 
	\indent
	Data Access Statement: No data was used for the research described in this article.}

\date{\today}

\bibliographystyle{amsalpha}

\begin{abstract}
	
Several constructions have been given for families of simple braces, but few examples are known of simple skew braces which are not braces. In this paper, we exhibit the first example of an infinite family of simple skew braces which are not braces and which do not arise from nonabelian simple groups. More precisely, we show that, for any primes $p$, $q$ such that $q$ divides ${(p^p-1)}/{(p-1)}$, there are exactly two simple skew braces (up to isomorphism) of order $p^p q$. 
\end{abstract}

\maketitle

\section{Introduction}

Braces were introduced by Rump \cite{Rump} in order to study nondegenerate involutive set-theoretic solutions of the Yang-Baxter equation. They were later generalised by Guarnieri and Vendramin \cite{GV}, who defined skew braces and showed that these give rise to set-theoretic solutions of the Yang-Baxter equation which are nondegenerate but not necessarily involutive. (Skew) braces have been investigated intensively, and connections have been found to many other topics in algebra including Hopf-Galois theory \cite{SV}, semigroups of $I$-type \cite{GI-vdB}, triangular semisimple Hopf algebras \cite{EG} and racks \cite{AG}.

A (left) skew brace $(B,\cdot, \circ)$ consists of a set $B$ with two binary operations such that $(B,\cdot)$ and $(B,\circ)$ are both groups, and a certain compatibility condition holds (see Definition \ref{skewbrace} below). We will usually write $B$ in place of $(B,\cdot,\circ)$. The order of $B$ is the cardinality $|B|$ of the underlying set. We shall refer to $(B,\cdot)$ as the additive group of $B$, and to $(B,\circ)$ as the multiplicative group. 
The skew brace $B$ is a brace if $(B,\cdot)$ is abelian.

There is a natural construction of quotient (skew) braces, and one can therefore consider simple (skew) braces. In the case of {\em braces}, Bachiller \cite{Bach-simple} showed how to obtain an infinite family of simple braces by a matched product construction. This approach has been generalised in several ways \cite{BCJO-iterated, BCJO-asym, CJO-simple} so that many examples of simple braces are now known. Indeed, every finite abelian group occurs as a subgroup of the additive group of a simple brace \cite{CJO-fabsim}. For a prime $p$, the only simple brace of $p$-power order is the unique brace of order $p$ \cite[Corollary to Proposition 8]{Rump}. Smoktunowicz \cite[Theorem 5]{Smok18} has given conditions on the prime factorisation of a natural number $n$ which are necessary for the existence of a simple brace of order $n$. Thus there are many values of $n$ for which it is known that no simple brace of order $n$ exists. However, apart from these negative results, there do not seem to be any composite $n$ for which a classification of all simple braces of order $n$ is known. 

In this paper, we will be concerned with simple {\em skew braces}. Of course, any simple brace is also a simple skew brace. Moreover, there are examples of simple skew braces which arise from nonabelian simple groups (see \S\ref{nonabsim} below). As with simple braces, there are no simple skew braces of order $p^r$ for $p$ prime and $r>1$ \cite[Corollary 4.9]{CSV}. This leaves open the possibility of finite simple skew braces which are not braces, but whose additive and multiplicative groups are both soluble. Indeed, such simple skew braces do exist. Konovalov and Vendramin constructed a database of skew braces of order $n$ for $n \leq 63$, $n \neq 32$, $48$, $54$, and this contains exactly two skew braces with soluble but nonabelian additive and multiplicative groups. These both have order $12$ \cite[Proposition 4.4]{KSV}.

The results of this paper can be summarised as follows. For each natural number $m$, we write $C_m$ for the cyclic group of order $m$.

\begin{theorem} \label{main}
Let $p$ and $q$ be primes such that $q$ divides $(p^p-1)/(p-1)$, and let $n=p^p q$. Then there are, up to isomorphism, exactly two simple skew braces $B$ of order $n$. These are mutually opposite. They have 
additive group 
$(B,\cdot) \cong V \rtimes C_q$
where $V \cong C_p^p$ is an elementary abelian group of order $p^p$ with $C_q$ acting nontrivially on $V$,
and they have multiplicative group 
$(B,\circ) \cong C_q \rtimes P$
where $P$ is a certain group of order $p^p$ acting on $C_q$ via an automorphism of order $p$. The group $\Autsb(B)$ of skew brace automorphisms of $B$ is cyclic of order $p$.
\end{theorem} 

The first few cases are $n=12$ with $(p,q)=(2,3)$; $n=351$ with $(p,q)=(3,13)$; $n =  34\, 374$ with $(p,q)=(5,11)$; and $n=221\, 875$ with $(p,q)=(5,71)$.

Theorem \ref{main} on the one hand shows the existence of an infinite family of pairs of simple skew braces, the first member of which is the pair of order $12$ found by Konovalov and Vendramin, and on the other hand (and in contrast to what is known for simple braces) gives a complete classification of simple skew braces whose order $n$ is of the special form considered.  

The paper is organised as follows. 

In \S2, we recall some preliminary definitions and results. This includes the description of a skew brace $(B,\cdot,\circ)$ in terms of a regular subgroup of the holomorph $\Hol(N,\cdot)=N \rtimes \Aut(N,\cdot)$ of its additive group. It is this description which underlies the link between skew braces and Hopf-Galois theory, and it will be essential to our approach in this paper. We also explain how the opposite of a skew brace can be identified in this description (Lemma \ref{op-lemma}), and we briefly review known results about simple skew braces arising from nonabelian simple groups (\S\ref{nonabsim}). 

In \S3, we construct a simple skew brace of order $n=p^p q$ as in Theorem \ref{main} by specifying the group $N$ of order $n$ which will be its additive group and then exhibiting a specific regular subgroup $G$ in $\Hol(N)$. We show in Theorem \ref{simple} that the skew brace $B$ corresponding to this group $G$ is simple. The techniques we use for calculating in $\Hol(N)$ have been developed in the context of Hopf-Galois theory, see for example \cite{pq, BC}. We also give an explicit description of a regular subgroup $G^*$ in $\Hol(N)$ corresponding to the opposite skew brace $B^\opp$. 

In \S4, we show (Theorem \ref{no-isom}) that the simple skew braces $B$ and $B^\opp$ are not isomorphic, so that we have indeed found two distinct simple skew braces of order $n$.

The method of proof of Theorem \ref{no-isom} is modified in \S5 to describe $\Autsb(B)$ in Theorem \ref{automs}. 

Finally, in \S6, we show in Theorem \ref{uniqueness} that every simple skew brace of order $n$ must be isomorphic to $B$ or $B^\opp$. Theorems \ref{simple}, \ref{no-isom}, \ref{automs} and \ref{uniqueness} together give all the assertions of Theorem \ref{main}.  

\newpage

\section{Preliminaries}

\subsection{Basic definitions}

\begin{definition} \label{skewbrace} 
A skew brace $(B,\cdot,\circ)$ consists of a set $B$ and two binary operations such that $(B, \cdot)$, $(B, \circ)$ are groups and
the brace relation
$$ a \circ (b \cdot c) = (a \circ b) \cdot a^{-1}  \cdot (a\circ c) 
\mbox{ for } a, b, c \in B $$
holds, where $a^{-1}$ denotes the inverse of $a$ in $(B,\cdot)$. We will usually just write $B$ instead of $(B,\cdot,\circ)$.

A brace is a skew brace $B$ in which $(B, \cdot)$ is abelian.

A homomorphism of skew braces $\theta:(B,\cdot,\circ) \to (B',\cdot',\circ')$ is a function $B \to B'$ with $\theta(a \cdot b) = \theta(a) \cdot' \theta(b)$ and $\theta(a \circ b) = \theta(a) \circ' \theta(b)$ for all $a$, $b \in B$. We write $\Autsb(B)$ for the group of skew brace automorphisms of $B$.
\end{definition}

If $B$ is a skew brace, there is a homomorphism of groups $\lambda: (B,\circ) \to \Aut(B,\cdot)$ given by
$$ a \mapsto \lambda_a \mbox{ where } \lambda_a(b) = a^{-1} \cdot (a \circ b). $$

\begin{definition} \label{def-left-ideal}
A left ideal in a skew brace $B$ is a subgroup $I$ of $(B,\cdot)$ such that
$\lambda_a(I) \subseteq I$ for all $a \in B$. It is then also a subgroup of $(B,\circ)$. 
An ideal in $B$ is a left ideal which is a normal subgroup in both $(B,\cdot)$ and $(B,\circ)$.
\end{definition}

\begin{remark} \label{char}
If $I$ is a characteristic subgroup in $(B,\cdot)$, it is automatically a left ideal which is normal in $(B,\cdot)$. If it is also normal in $(B,\circ)$ then it is an ideal. 
\end{remark}

If $I$ is an ideal in a skew brace $B$, we can form the quotient skew brace $B/I$. 

\begin{definition}
A skew brace $B$ is simple if $|B|>1$ and $B$ has no ideals except $\{1\}$ and $B$. 
\end{definition}

\subsection{Connection with the holomorph} \label{Hol-cor}

We define the holomorph of a group $N$ as the semidirect product
$$ \Hol(N) = N \rtimes \Aut(N) , $$
and we view this as a group of permutations of (the underlying set of) $N$. Thus if $(\eta,\alpha) \in \Hol(N)$ with $\eta \in N$, $\alpha \in \Aut(N)$, then we have 
$$ (\eta, \alpha) \mu = \eta  \alpha(\mu) \mbox{ for } \mu \in N, $$
so that, in particular, $(\eta,\alpha)1_N=\eta$, where $1_N$ is the identity element of $N$. 
Multiplication in $\Hol(N)$ is given by
$$ (\eta,\alpha) (\mu, \beta) = (\eta \alpha(\mu), \alpha \beta). $$
A subgroup $G$ of $\Hol(N)$ is regular if, given any $\eta$, $\mu \in N$, there is a unique $g \in G$ with $g \eta = \mu$. 

For a skew brace $(B,\cdot, \circ)$, the maps $\lambda_a$ allow us to embed $(B,\circ)$ into $\Hol(B,\cdot)$ by
$$ b \mapsto (b, \lambda_b). $$
The image of this embedding is a regular subgroup $G \cong (B,\circ)$ of $\Hol(B,\cdot)$. Conversely, let $(N,\cdot)$ be a group and let $G$ be a regular subgroup of $\Hol(N)$. Then, for each $\eta \in N$, there is a unique element $g_\eta=(\eta, \alpha_\eta) \in G$ whose first component is $\eta$, and we have a skew brace $(N,\cdot, \circ)$ where $\circ$ is defined by 
\begin{equation}  \label{def-circ}
	g_{\eta \circ \eta'} = g_\eta g_{\eta'}. 
\end{equation}
Two regular subgroups $G$, $G'$ in $\Hol(N)$ give isomorphic skew braces if and only if they are conjugate in $\Hol(N)$ via an element of $\Aut(N)$ \cite[Proposition 4.3]{GV}. 

\subsection{Opposite skew braces}  \label{opposite}

\begin{definition} \cite{KT}
	If $(B, \cdot, \circ)$ is a skew brace, then its opposite skew brace $B^\opp$ is $(B, \cdot^\opp,\circ)$ where $a \cdot^\opp b = b \cdot a$. 
\end{definition}

It is easy to see that $\Autsb(B^\opp)=\Autsb(B)$.

\begin{lemma} \label{op-lemma}
Let $N$ be a finite group, and let $G=\{g_\eta: \eta \in N\}$ be a regular subgroup of $\Hol(N)$, where $g_\eta=(\eta, \alpha_\eta)$. Define 
$$ g^*_\eta = (\eta^{-1}, \conj(\eta) \alpha_\eta )\in \Hol(N) $$
where $\conj(\eta)$ is the inner automorphism  of $N$ given by conjugation by $\eta$. Let $G^*=\{g_\eta^* : \eta \in N\}$. Then $G^*$ is a regular subgroup of $\Hol(N)$, and the skew brace determined by $G^*$ is the opposite skew brace to that determined by $G$. Moreover $G$ and $G^*$ are isomorphic via $g_\eta \mapsto g^*_\eta$.
\end{lemma}
\begin{proof}
Since $(\eta \cdot \eta')^{-1} = \eta^{-1} \cdot^\opp \eta'^{-1}$, it suffices by (\ref{def-circ}) to show that if $g_{\eta} g_{\eta'} = g_\mu$ then $g^*_{\eta} g^*_{\eta'} = g^*_\mu$. Now
$$ g_\eta g_{\eta'} = (\eta \alpha_{\eta}(\eta') , \alpha_{\eta} \alpha_{\eta'}), $$
so
$$ \mu = \eta \alpha_{\eta}(\eta'), \qquad \alpha_\mu=  \alpha_{\eta} \alpha_{\eta'}. $$
We have
$$ g^*_{\eta} g^*_{\eta'} 
  =  (\eta^{-1} \bigl( \conj(\eta) \alpha_{\eta} \bigr) (\eta'^{-1}) , 
	\conj(\eta) \alpha_{\eta}  \conj(\eta') \alpha_{\eta'} \bigr) . $$
The first component is 
$$ \eta^{-1} \eta \alpha_{\eta} (\eta')^{-1} \eta^{-1}
	= \mu^{-1}, $$
while the second is the automorphism taking $\nu \in N$ to 
\begin{eqnarray*}
	\eta \alpha_{\eta} \bigl( \eta' \alpha_{\eta'}(\nu)\eta'^{-1} \bigr) \eta^{-1} 
	& = & \eta \alpha_{\eta}(\eta') \alpha_\eta(\alpha_{\eta'}(\nu)) \alpha_\eta(\eta')^{-1}  \eta^{-1} \\
	& = & \mu \alpha_\mu(\nu) \mu^{-1} \\
	& = & \conj(\mu) \alpha_\mu(\nu).
\end{eqnarray*}
Hence 
$$ g^*_{\eta} g^*_{\eta'} = (\mu^{-1}, \conj(\mu) \alpha_\mu) = g^*_\mu. $$
\end{proof}

\subsection{Simple skew braces coming from nonabelian simple groups}
\label{nonabsim}
Since an ideal in a brace $(B,\cdot, \circ)$ is a normal subgroup in each of the groups $(B,\cdot)$ and $(B,\circ)$, it is clear that if either of these groups is a simple group then $B$ must be a simple skew brace. 

In the context of Hopf-Galois theory, it has been shown \cite{simple} that if the holomorph $\Hol(N)$ of some finite group $N$ contains a regular subgroup $G$ which is a nonabelian simple group then $N \cong G$ and $G$ is the image of either the left or right regular representation of $N$. This means that if $B$ is a finite skew brace for which $(B,\circ)$ is a nonabelian simple group then $B$ is either the trivial skew brace $(B,\circ,\circ)$ or its opposite skew brace $(B,\circ^\opp, \circ)$.

On the other hand, there are finite skew braces where $(B,\cdot)$ is a nonabelian simple group but $(B,\circ) \not \cong (B, \cdot)$. We will say that $(B,\cdot)$ has an exact factorisation $A \cdot C$ if $A$ and $C$ are subgroups of $(B,\cdot)$ such that every element of $B$ can be written uniquely in the form $a \cdot c$ with $a \in A$ and $c \in C$. By \cite[Theorem 2.3]{SV}, any such exact factorisation gives rise to a skew brace with additive group $B$ and multiplicative group $A \times C$. For example, writing $A_n$ for the alternating group of degree $n$ and $C_m$ for the cyclic group of order $m$, the exact factorisation $A_5=A_4 \cdot C_5$ gives a skew brace whose additive group is the nonabelian simple group $A_5$ and whose multiplicative group is the soluble group $A_4 \times C_5$. Tsang \cite{Tsang}  has determined all cases where the additive group is a nonabelian simple group and the multiplicative group is soluble. 

\section{Construction of our simple skew braces}

\begin{notation} \label{pq}
For the rest of the paper, we fix primes $p$ and $q$ such that $q$ divides $(p^p-1)/(p-1)$, and we fix 
$n=p^p q$, as in Theorem \ref{main}. 
\end{notation}

Since $\gcd(p-1,p^p-1)=p-1$, the condition on $p$ and $q$ holds if and only if $q$ divides $p^p-1$ but does not divide $p-1$. Our hypotheses mean that the multiplicative order of $p$ mod $q$ is $p$, so in particular $p$ divides $q-1$. It is possible that $p^2$ divides $q-1$, the smallest instance of this being $p=59$, $q=141\,579\,233$. (The author is indebted to Andrew Darlington for finding this example.) 

In this section, we will construct a simple skew brace $(B, \cdot, \circ)$ of order $n=p^p q$ as a regular subgroup in the holomorph of certain group $N$. Thus $(B,\cdot) \cong N$. We will then make explicit the operation $\circ$, and give a similar description of the opposite skew brace $B^\opp$. We fix the following notation, which will be in force for the rest of the paper.

\begin{notation}  \label{fix-K}
Let $V=\F_p^p$ be the space of $p$-component column vectors over the field $\F_p$ of $p$ elements. Then we may view $V$ as an elementary abelian group of order $p^p$, and also as a left module for the matrix ring $\Mat_p(\F_p)$. Further, let $K=\F_{p^p}$ be the field of $p^p$ elements. Given a matrix $M \in \GL_p(\F_p)$ whose minimal polynomial over $\F_p$ is irreducible and of degree $p$, we may identify the subring $\F_p[M]$ of $\Mat_p(\F_p)$ with $K$, and so view $V$ as a $1$-dimensional vector space over $K$. We will variously view $V$ as a space of column vectors, an abelian group, or a $K$-vector space, as convenient. 
\end{notation}

\subsection{Characterisation of the additive group}

The additive group $N$ of the simple skew brace we will construct will be a group of order $n=p^p q$ containing a normal Sylow $p$-subgroup $P$. In any such group $N$, any Sylow $q$-subgroup $Q$ is a complement to $P$, so that $N$ is a semidirect product $N \cong P \rtimes Q$ for some action of $Q \cong C_q$ on $P$. The next results shows that, if $Q$ is not also normal on $N$, then there is exactly one possibility for $N$.

\begin{lemma} \label{constr-N}
Up to isomorphism, there is a unique group $N$ of order $n$ which has a normal Sylow $p$-subgroup $P$ but does not have a normal Sylow $q$-subgroup. It has 
the following (noncanonical) description. Let $M \in \GL_p(\F_p)$ be a matrix of order $q$. Then $N$ is the group of block matrices in $\GL_{p+1}(\F_p)$ of the form
$$  \left( \begin{array}{c|c} M^k & v \\ \hline 0 & 1 \end{array} \right) \mbox{ for }   0 \leq k \leq q-1 \mbox{ and } v \in V, $$
and $P$ is the subgroup of such matrices with $k=0$.    
\end{lemma}
\begin{proof}
We first note that the existence of a matrix $M$ of order $q$ is guaranteed. The multiplicative group $K^\times $ of $K$ is cyclic of order $p^p-1$, so, by our hypotheses in Notation \ref{pq},  $K$ contains an element $\alpha$ of order $q$ and $\F(\alpha)=K$. The minimal polynomial $f(X)$ of $\alpha$ over $\F_p$ is then irreducible of degree $p$, and any matrix $M \in \GL_p(\F_p)$ with characteristic polynomial $f(X)$, for example the companion matrix of $f(X)$, has order $q$. 

Now let $M \in \GL_p(\F_p)$ be any matrix of order $q$. As $q$ is prime, $M$ has an eigenvalue $\alpha$ of order $q$ in some extension of $\F_p$, and our hypotheses ensure that $\F_p(\alpha)=K$. The minimal polynomial of $\alpha$ over $\F_p$ is irreducible of degree $p$ and therefore coincides with the minimal polynomial of $M$. Starting with this matrix $M$, we obtain a group $N$ of  
order $p^p q$ consisting of block matrices, as in the statement of the Lemma. The elements of $N$ with $k=0$ form a normal Sylow $p$-subgroup $P$. We identify this subgroup with $V$. The elements with $v=0$ form a Sylow $q$-subgroup $Q$ which is not normal in $N$. 

Finally, let $H$ be any group of order $n$ which has a normal Sylow $p$-subgroup $P$. Let $Q$ be a Sylow $q$-subgroup, and assume that $Q$ is not normal in $H$. We will show that $H \cong N$. (In particular, this will show that, up to isomorphism, the group $N$ does not depend on the choice of $M$.) Our hypotheses on $H$ ensure that $H = P \rtimes Q$ where  $Q$ acts nontrivially on $P$, so conjugation by a generator $\sigma$ of $Q$ induces an automorphism of $P$ of order $q$. Let $F$ be the Frattini subgroup of $P$. Then $P/F$ is an elementary abelian group of order $p^r$ for some $r \leq p$. By \cite[5.3.3]{Rob}, $|\Aut(P)|$ divides $p^{(p-r)r} |\GL_r(\F_p)|$. Since $q \nmid |\GL_s(\F_p)|$ when $s<p$, it follows that $r=p$ and $F$ is trivial. Hence we can (and do) identify $P$ with $\F_p^p$, and $\sigma$ acts on $P$ as some matrix $M_1 \in \GL_p(\F_p)$ of order $q$. Identifying $\F_p[M_1]$ with $K$, we then have that $\sigma$ acts on $P$ as multiplication by some element $\beta \in K^\times$ of order $q$. Replacing $\sigma$ by a suitable power, we may assume that $\beta$ coincides with the eigenvalue $\alpha$ of the matrix $M$ used to construct $N$. Then $M_1$ and $M$ have the same minimal polynomial, which is irreducible of degree $p$. It follows that $M_1$ and $M$ are conjugate in $\GL_p(\F_p)$. Making a change of basis, we may therefore assume that $M_1=M$. Hence $H \cong N$.
\end{proof}

\begin{lemma} \label{N-subgroups}
In the group $N$ constructed in Lemma \ref{constr-N}, the Sylow $p$-subgroup $P$ is the unique nontrivial proper normal subgroup in $N$. Also, $N$ has no subgroup of order $p^i q$ for $0<i <p$.  
\end{lemma}
\begin{proof}
We again view $P$ as a $1$-dimensional vector space over $K \cong \F_p[M]$, so that conjugation by $M$ corresponds to multiplication by an element $\alpha \in K$ of order $q$. Then $\F_p(\alpha)=K$. If $L \lhd N$ then $L \cap P \lhd N$, and $L \cap P$ is a $K$-subspace of $P$. Thus $L \cap P=0$ or $P$. If $L \cap P=0$ then $L$ is trivial since $N$ has no normal subgroup of order $q$. If $L \cap P=P$ then $L=P$ or $N$. This proves the first statement. For the second statement, if $|H|=p^i q$ then $H \cap P \lhd H$, and some element of order $q$ in $H$ acts on $H \cap P$ as multiplication by $\alpha$. Again $H \cap P$ is a $K$-subspace of $P$, so $i=0$ or $p$.  
\end{proof}

We next record a number of facts about the the matrix $M$ which will be needed later in the paper.  

\begin{lemma} \label{centr-norm}   
Let $M \in \GL_p(\F_p)$ be any matrix of order $q$.
\begin{itemize}
\item[(i)]
The only $M$-invariant $\F_p$-subspaces of $V$ are $0$ and $V$. 
\item[(ii)]
$M-I$ is invertible.
\item[(iii)]
If $w \in V$ with $w \neq 0$ then the vectors $(M^i-I)w$ for $i \in \{1, 2, \ldots,p\}$ span $V$. 
\item[(iv)]
The centraliser in $\GL_p(\F_p)$ of the subgroup $\langle M \rangle$ is $\F_p[M]^\times$, which is cyclic of order $p^p-1$, generated by a matrix $T$ with $T^{(p^p-1)/q}=M$.
\item[(v)] If $A \in \GL_p(\F_p)$ and $AM=MA$, then the only $A$-invariant subspaces of $V$ are  $0$ and $V$ unless $A$ is a scalar matrix $\lambda I$ for some $\lambda \in \F_p^\times$.
\item[(vi)] There is a matrix $J \in \GL_p(\F_p)$ such that $JMJ^{-1}=M^p$ and $J^p=I$. Any such $J$ has minimal polynomial $X^p-1=(X-1)^p$ over $\F_p$.
\item[(vii)] The normaliser in $\GL_p(\F_p)$ of the subgroup $\langle M \rangle$ is generated by $T$ and $J$ and has order $(p^p-1)p$. Any element $A$ of this normaliser can be written $A=T^i J^j$ with $0 \leq i < p^p-1$ and $0 \leq j <p$. Then $AMA^{-1}=M^{p^j}$. Also, $A$ has order $p$ if and only if $i$ is divisible by $p-1$ and $i \neq 0$, $j \neq 0$, and $A$ has order $q$ if and only if $A=M^r$ with $1 \leq r \leq q-1$.
\end{itemize}

\end{lemma}
\begin{proof}
(i)	We again identify $\F_p[M]$ with $K$, so that conjugation by $M$ corresponds to multiplication by an element $\alpha \in K$ of order $q$. If $W$ is an $M$-invariant $\F_p$-subspace of $V$ then $\alpha W=W$. As $\F_p(\alpha)=K$, we must have $W=0$ or $V$. 

\noindent (ii) By (i), the subspace $\{v \in V: Mv=v\}$ is trivial. Hence $M-I$ is invertible.

\noindent (iii) Since the minimal polynomial of $M$ over $\F_p$ has degree $p$, it is clear that the matrices $I$, $M+I$, $M^2+M+I$, \ldots, $M^{p-1}+ \cdots +M+I$ form a basis of the $\F_p$-algebra $\F_p[M]$. Multiplying by the invertible matrix $M-I$, we find that $M-I$, $M^2-I$, \ldots, $M^p-I$ is also a basis. Thus the $(M^i-I)w$ span $V$.

\noindent (iv) By the Double Centraliser Theorem \cite[\S12.7]{Pierce}, the centraliser of the subalgebra $\F_p[M]$ in $\Mat_p(\F_p)$ is $\F_p[M]$ itself. Hence the centraliser of $\langle M \rangle$ in $\GL_p(\F_p)$ is $\F_p[M]^\times \cong K^\times$, which is cyclic of order $p^p-1$. As $M$ has order $q$, there is some generator $T$ with $T^{(p^p-1)/q}=M$.  

\noindent (v) We have $A=T^i$ for some $i$, and $A$ acts on $V$ as multiplication by some $\lambda \in K^\times$. If $\lambda \not \in \F_p$ then $\F_p(\lambda)=K$ and the only $A$-invariant subspaces are $0$ and $V$. 

\noindent (vi) By the Noether-Skolem Theorem \cite[\S12.6]{Pierce}, every automorphism of $\F_p[M]$ is induced by conjugation by some element of $\GL_p(\F_p)$. In particular, there is some matrix $J$ inducing the Frobenius automorphism $\phi: x \mapsto x^p$. Thus $JMJ^{-1}=M^p$. Then $J^p$ centralises $\langle M \rangle$, so $J^p$ has order dividing $p^p-1$ by (iv).  
Replacing $J$ by $J^{1-p^p}$, we may therefore assume that $J^p=I$. Clearly $J$ has characteristic polynomial $X^p-1=(X-1)^p$ over $\F_p$, so it only remains to verify that $(J-I)^{p-1} \neq 0$. Now $(J-I)^{p-1}=J^{p-1}+ \cdots+ J +I$ acts on $\F_p[M]$ as the trace operator $\phi^{p-1} + \cdots + \phi +1$ from $K$ to $\F_p$. Since the trace is nondegenerate, we have $(J-I)^{p-1} \neq 0$ as required. 

\indent (vii) If $AMA^{-1}=M^s$ then conjugation by $A$ induces an automorphism of the $\F_p$-algebra $\F_p[M] \cong K$. Since the Galois group of $K/\F_p$ is cyclic of order $p$, generated by $\phi$, we must have $AMA^{-1}=M^{p^j}$ with $0 \leq j \leq p-1$. Then $AJ^{-j}$ centralises $\langle M \rangle$, so, by (iv), $A=T^i J^j$ for some $i$, and $AMA^{-1}=A^{p^j}$. Hence the normaliser is generated by $T$ and $J$ and has order $(p^p-1)p$. We have 
$$ A^k = T^{i(1+p^j+p^{2j} + \cdots + p^{(k-1)j})} J^{jk}. $$
For $A$ to have order $p$, we need $j \neq 0$ and $i(1+p^j+p^{2j} + \cdots + p^{(k-1)j}) \equiv 0 \pmod{p^p-1}$. But for $j \neq 0$, the residue classes $p^j$, $p^{2j}$,\ldots ,$p^{(p-1)j}$ mod $p^p-1$ are congruent in some order to $1$, $p$, $p^2$, \ldots, $p^{p-1}$, so $A$ has order dividing $p$ if and only if $i(1+p+ \cdots + p^{p-1}) \equiv 0 \pmod{p^p-1}$, which holds precisely when $p-1$ divides $i$. For $A$ to have order $q$ we need $j=0$, so that $A^k=T^{ik}$, with $i \neq 0$ and $i=(p^p-1)r/q$ for some $r$. Then $A=M^r$. 
\end{proof}

\begin{remark}  \label{J-Jordan}
Any matrix in $\GL_p(\F_p)$ with minimal polynomial $X^p-1$ is conjugate to the Jordan matrix 
\begin{equation} \label{Jordan} 
		J = \begin{pmatrix} 1 & 1 & 0 & 0 & \ldots & 0 & 0 \\
			0 & 1 & 1 & 0 & \ldots & 0 & 0 \\
			0 & 0 & 1 & 1 & \ldots & 0 & 0 \\
			\vdots &  \vdots & \vdots &  \vdots &        & \vdots & \vdots \\
				0 & 0 & 0 & 0 & \ldots & 1 & 1 \\
				0 & 0 & 0 & 0 & \ldots & 0 & 1 \end{pmatrix}. 
\end{equation} 
so, by a change of basis, we choose the matrix $M$ of order $q$ so that the matrix $J$ in Lemma \ref{centr-norm}(vi) is as in (\ref{Jordan}).
\end{remark}

We next describe the automorphisms of the group constructed in Lemma \ref{constr-N}.

\begin{lemma} \label{Aut-N}
$\Aut(N)$ can be identified with the group of block matrices
$$	\left\{ \left( \begin{array}{c|c} A & v \\ \hline 0 & 1 \end{array} \right) :  A \in \GL_p(\F_p) \mbox{ normalises } \langle M \rangle, \, v \in V   \right\} $$
acting on $N$ by conjugation in $\GL_{p+1}(\F_p)$. 
\end{lemma}
\begin{proof}
Let $\theta \in \Aut(N)$. As $P$ is characteristic in $N$, there is some $A \in \GL_p(\F_p)$, uniquely determined by $\theta$, such that
$$ \theta \left( \left( \begin{array}{c|c} I & w \\ \hline 0 & 1 \end{array} \right) \right)
= \left(  \begin{array}{c|c} I & Aw \\ \hline 0 & 1 \end{array} \right) $$
for all $w \in V$. Also
$$ \theta \left(  \left( \begin{array}{c|c} M & 0 \\ \hline 0 & 1 \end{array} \right) \right)
= \left(  \begin{array}{c|c} M^k & u \\ \hline 0 & 1 \end{array} \right) 
$$   
for some $k \in \{1, 2, \ldots, q-1\}$ and some $u \in V$. Applying $\theta$ to the relation
$$ \left(  \begin{array}{c|c} M & 0 \\ \hline 0 & 1 \end{array} \right) 
\left(  \begin{array}{c|c} I & w \\ \hline 0 & 1 \end{array} \right) 
= \left(  \begin{array}{c|c} I & Mw \\ \hline 0 & 1 \end{array} \right) 
\left(  \begin{array}{c|c} M & 0 \\ \hline 0 & 1 \end{array} \right), $$
we have
$$ \left(  \begin{array}{c|c} M^k & u \\ \hline 0 & 1 \end{array} \right) 
\left(  \begin{array}{c|c} I & Aw \\ \hline 0 & 1 \end{array} \right) 
= \left(  \begin{array}{c|c} I & AMw \\ \hline 0 & 1 \end{array} \right) 
\left(  \begin{array}{c|c} M^k & u \\ \hline 0 & 1 \end{array} \right). $$ 
This holds for all $w$ if and only if $AMA^{-1}=M^k$. Hence $A$ normalises $\langle M \rangle$. The matrix $I-AMA^{-1}=A(M-I)A^{-1}$ is invertible by Lemma \ref{centr-norm}(ii). Thus there is a unique $v \in V$ with $(I-AMA^{-1})v=u$. Then $\theta$ corresponds to conjugation by the block matrix
$$  \left(  \begin{array}{c|c} A & v \\ \hline 0 & 1 \end{array} \right).  $$ 
Conversely, for any $A \in \GL_p(\F_p)$ normalising $\langle M \rangle$ and any $v \in V$, we easily verify that conjugation by the above matrix gives an automorphism of $N$. 
\end{proof}

We record that
$$ 	 \left( \begin{array}{c|c} A & v \\ \hline 0 & 1 \end{array} \right)^{-1} = 
   	 \left( \begin{array}{c|c} A^{-1} & -A^{-1}v \\ \hline 0 & 1 \end{array} \right) $$
and, for $k \geq 0$, 
$$  \left( \begin{array}{c|c} A & v \\ \hline 0 & 1 \end{array} \right)^k = 
     \left(\begin{array}{c|c} A^k & A^{[k]} v \\ \hline 0 & 1 \end{array} \right) $$
where we write
\begin{equation} \label{def-bk}    
	A^{[0]}=0, \qquad A^{[k]} = \sum_{i=0}^{k-1} A^i \mbox{ for } k \geq 1.
\end{equation}
We will use these facts without further comment.

\subsection{Construction of a simple skew brace}

From now on, $N$ denotes the group of order $p^p q$ constructed in Lemma \ref{constr-N}, where $M$ and $J$ are chosen as in Remark \ref{J-Jordan}. In particular, $J$ is the Jordan matrix (\ref{Jordan}). 
Let $e_1$, \ldots, $e_p$ be the standard basis of $V=\F_p^p$.

By Lemma \ref{Aut-N}, we can write elements of $\Hol(N)$ in the form $[R, \conj(S)]$ where 
$$ R =  \left( \begin{array}{c|c} M^k & v \\ \hline 0 & 1 \end{array} \right), \qquad 
   S=  \left( \begin{array}{c|c} A & w \\ \hline 0 & 1 \end{array} \right)  $$
with $0 \leq k \leq q-1$, $AMA^{-1}=M^s$ for some $s$, and $v$, $w \in V$.   
We include the notation $\conj$ in elements of $\Hol(N)$ to emphasise that the action of $S$ on $N$ is by conjugation rather than simply by matrix multiplication. Thus multiplication in $\Hol(N)$ is given by 
$$ [R_1, \conj(S_1)] [R_2, \conj(S_2) ] = [R_1 S_1 R_2 S_1^{-1}, \conj(S_1 S_2)]. $$

We consider the following elements of $\Hol(N)$:

$$ X = \left[  \left( \begin{array}{c|c} M  & 0 \\ \hline 0 & 1 \end{array} \right) , \conj \left( \begin{array}{c|c} I &  0 \\ \hline  0 & 1 \end{array} \right) \right], $$
$$ Y_v = \left[ \left( \begin{array}{c|c} I & v \\ \hline 0 & 1 \end{array} \right) , 
\conj \left( \begin{array}{c|c} I &  -v \\ \hline  0 & 1 \end{array} \right) \right] \mbox{ for each } v \in V, $$ 

$$ Z= \left[ \left( \begin{array}{c|c} I & e_p \\ \hline 0 & 1 \end{array} \right) , 
\conj \left( \begin{array}{c|c} J &  -e_p \\ \hline  0 & 1 \end{array} \right) \right]. $$
In particular, $Y_0=1$, the identity element of $\Hol(N)$.


We will record in Lemma \ref{G-rels} various relations satisfied by these elements. We first make some preliminary observations. 

Clearly for $i \geq 0$ we have 
\begin{equation} \label{Xi}
	X^i = \left[  \left( \begin{array}{c|c} M^i  & 0 \\ \hline 0 & 1 \end{array} \right) , \conj \left( \begin{array}{c|c} I &  0 \\ \hline  0 & 1 \end{array} \right) \right].
\end{equation}

To obtain a formula for $Z^k$, we first note that routine induction arguments give
\begin{equation} \label{J-k}
	 J^k e_p = \sum_{i=0}^k {k \choose i} e_{p-i}, 
\end{equation} \mbox{ and } 
\begin{equation} \label{J-bk} 
	  J^{[k]} e_p = \sum_{h=0}^{k-1} {k \choose h+1} e_{p-h} 
	    = k e_p + {k \choose 2} e_{p-1} + \cdots + e_{p+1-k}.
\end{equation}

\begin{proposition}  \label{Zk}
For $k \geq 0$, we have
 $$Z^k = \left[ \left( \begin{array}{c|c} I & J^{[k]}e_p \\ \hline 0 & 1 \end{array} \right) , 
\conj \left( \begin{array}{c|c} J^k  &  - J^{[k]}e_p \\ \hline  0 & 1 \end{array} \right) \right]. $$ 
\end{proposition}
\begin{proof}
We again use induction: the stated formula holds for $k=0$, and if it holds for some $k \geq 0$, then
\begin{eqnarray*}
\lefteqn{	Z^{k+1} } \\
 & = &  \left[ \left( \begin{array}{c|c} I & e_p \\ \hline 0 & 1 \end{array} \right) , 
	\conj \left( \begin{array}{c|c} J &  -e_p \\ \hline  0 & 1 \end{array} \right) \right] \;
	\left[ \left( \begin{array}{c|c} I & J^{[k]}e_p \\ \hline 0 & 1 \end{array} \right) , 
	\conj \left( \begin{array}{c|c} J^k  &  - J^{[k]}e_p \\ \hline  0 & 1 \end{array} \right) \right] \\
	& = &  \left[ A , \conj \left( \begin{array}{c|c} J^{k+1} &  - J^{[k+1]} e_p \\ \hline  0 & 1 \end{array} \right) \right] 
\end{eqnarray*}
where
\begin{eqnarray*}
	A & = & \left( \begin{array}{c|c} I & e_p \\ \hline 0 & 1 \end{array} \right) 
	\left( \begin{array}{c|c} J &  -e_p \\ \hline  0 & 1 \end{array} \right)
	\left( \begin{array}{c|c} I & J^{[k]}e_p \\ \hline 0 & 1 \end{array} \right) 
	\left( \begin{array}{c|c} J^{-1} &  J^{-1} e_p \\ \hline  0 & 1 \end{array} \right) \\
	& = & 	\left( \begin{array}{c|c} J &  0 \\ \hline  0 & 1 \end{array} \right)
	\left( \begin{array}{c|c} J^{-1} & J^{-1}e_p + J^{[k]} e_p \\ \hline 0 & 1 \end{array} \right) \\
	& = & \left( \begin{array}{c|c} I & J^{[k+1]}e_p \\ \hline 0 & 1 \end{array} \right). 
\end{eqnarray*}
\end{proof}

\begin{lemma}  \label{G-rels}
The elements $X$, $Y_v$, $Z$ satisfy the following relations for $v$, $w \in V$:
\begin{itemize}
	\item[(i)] $X^q=1$, 
	\item[(ii)] $Z^p=Y_{e_1}$.
	\item[(iii)] $ZX=X^p Z$, 
	\item[(iv)] $Y_v X = X Y_v$,
	\item[(v)] $Y_v Y_w = Y_{v+w}$;
	\item[(vi)] $Z Y_v = Y_{Jv} Z$.
\end{itemize}
\end{lemma}
\begin{proof}
(i) is immediate from (\ref{Xi}).

\noindent (ii) follows from the case $k=p$ of Proposition \ref{Zk} since $J^p=I$ and, using (\ref{J-bk}), $J^{[p]} e_p=e_1$. 

\noindent (iii) We calculate 
\begin{eqnarray*}
	Z X & = &  \left[ \left( \begin{array}{c|c} I & e_p \\ \hline 0 & 1 \end{array} \right) , 
	\conj \left( \begin{array}{c|c} J &  -e_p \\ \hline  0 & 1 \end{array} \right) \right] \;
	\left[ \left( \begin{array}{c|c} M & 0 \\ \hline 0 & 1 \end{array} \right) , 
	\conj \left( \begin{array}{c|c} I  & 0 \\ \hline  0 & 1 \end{array} \right) \right] \\
	& = &  \left[  \left( \begin{array}{c|c} I & e_p \\ \hline 0 & 1 \end{array} \right) 
	       \left( \begin{array}{c|c} J &  -e_p \\ \hline  0 & 1 \end{array} \right) 
	        \left( \begin{array}{c|c} M  & 0 \\ \hline  0 & 1 \end{array} \right)
	         \left( \begin{array}{c|c} J^{-1} &  J^{-1} e_p \\ \hline  0 & 1 \end{array} \right) ,
	         	\conj \left( \begin{array}{c|c} J  & -e_p \\ \hline  0 & 1 \end{array} \right) \right] \\
	& = &  \left[ \left( \begin{array}{c|c} JMJ^{-1} & JMJ^{-1}e_p \\ \hline 0 & 1 \end{array} \right) , 
	\conj \left( \begin{array}{c|c} J &  -e_p \\ \hline  0 & 1 \end{array} \right) \right]
\end{eqnarray*}
and, using (\ref{Xi}), 
\begin{eqnarray*}
	 X^p Z & = & 	\left[ \left( \begin{array}{c|c} M^p & 0 \\ \hline 0 & 1 \end{array} \right) , 
	 \conj \left( \begin{array}{c|c} I  & 0 \\ \hline  0 & 1 \end{array} \right) \right]  
	     \left[ \left( \begin{array}{c|c} I & e_p \\ \hline 0 & 1 \end{array} \right) , 
	\conj \left( \begin{array}{c|c} J &  -e_p \\ \hline  0 & 1 \end{array} \right) \right] \\
	& = &  \left[ \left( \begin{array}{c|c} M^p & M^pe_p \\ \hline 0 & 1 \end{array} \right) , 
	\conj \left( \begin{array}{c|c} J &  -e_p \\ \hline  0 & 1 \end{array} \right) \right].
\end{eqnarray*}
As $JMJ^{-1}=M^p$, we have $ZX =X^p Z$.

\noindent (iv), (v), (vi) are proved by similar calculations, which we leave to the reader. 	
\end{proof}
\begin{corollary}  \label{commutator}
	For each $v \in V$ we have $ZY_vZ^{-1}Y_v^{-1} = Y_{Jv-v}$. In particular, 
$$ Z Y_{e_i} Z^{-1} Y_{e_i}^{-1} = Y_{e_{i-1}} \mbox{ for } 2 \leq i \leq p.  $$
\end{corollary}

We now consider the subgroup
\begin{equation} \label{def-G}
	G=\langle X, Y_{e_{p-1}},Z\rangle 
\end{equation} 
of $\Hol(N)$.

\begin{notation}  \label{V0}
Let $V_0$ denote the $\F_p$-subspace of $V$ spanned by $e_1, \ldots, e_{p-1}$. 
\end{notation}

From (\ref{J-bk}) we have
\begin{equation}  \label{Jk-ep}
	J^{[k]}e_p - ke_p \in V_0 \mbox{ for } 0 \leq k \leq p-1. 
\end{equation}

\begin{lemma} \label{G-reg}
The group $G$ in (\ref{def-G}) is a regular subgroup of $\Hol(N)$. Moreover, every element $g$ of $G$ can be written uniquely in the form  
\begin{equation} \label{g1}
	 g = X^i Y_u Z^k \mbox{ with } 0 \leq i < q, \; u \in V_0, \; 0 \leq k < p, 
\end{equation}
or, equivalently, in the form
\begin{equation} \label{g2} 
	g = X^i Y_{e_1}^{k_1} \cdots Y_{e_{p-1}}^{k_{p-1}} Z^k \mbox{ with } 
	0 \leq i < q, \;  0\leq k_1, k_2, \ldots, k_{p-1}, k<p.
\end{equation}
\end{lemma}
\begin{proof}
We first consider the expressions (\ref{g1}) and (\ref{g2}). Applying Corollary \ref{commutator} repeatedly, with find that $Y_{e_{p-2}}, Y_{e_{p-3}}, \ldots, Y_{e_1} \in G$.
Using Lemma \ref{G-rels}(v), it follows that, for any element $u = k_1 e_1 + \cdots + k_{p-1} e_{p-1}$ of $V_0$, we have $Y_u = Y_{e_1}^{k_1} \cdots Y_{e_{p-1}}^{k_{p-1}} \in G$. This shows the equivalence of (\ref{g1}) and (\ref{g2}), and also that every element of the form (\ref{g1}) is in $G$. On the other hand, using the relations in Lemma \ref{G-rels}, we may rewrite the product of any two such elements in the form $X^i Y_v Z^k$ for some $i \geq 0$, $v \in V$ and $k \geq 0$. Since $Ju \in V_0$ for all $u \in V_0$, and $V_0$ is closed under addition, we will always have $v \in V_0$. By Lemma \ref{G-rels}(i), we can take $0 \leq i \leq q-1$. From Lemma \ref{G-rels}(ii), (v) we see that $Z$ has order $p^2$. However, if $k \geq p$, we may replace $Z^k$ by $Y_{e_1} Z^{k-p}$. Hence every element of $G$ can be written in the form (\ref{g1}). In particular, as the number of triples $(i,u,k)$ is $q \cdot p^{p-1} \cdot p$, we have $|G| \leq p^p q = |N|$.

A direct calculation using (\ref{Xi}) and Proposition \ref{Zk} shows the element $X^i Y_u Z^k$ in (\ref{g1}) can be written
\begin{equation}\label{XYZ}
	\left[ \left( \begin{array}{c|c} M^i & M^i (u + J^{[k]} e_p) \\ \hline 0 & 1 \end{array} \right) , 
	\conj \left( \begin{array}{c|c} J^k &  -u -J^{[k]} e_p \\ \hline  0 & 1 \end{array} \right) \right]. 
\end{equation}  	
For an arbitrary element 
$$ \eta = \left( \begin{array}{c|c} M^i & v \\ \hline 0 & 1 \end{array} \right) $$
of $N$, it follows from (\ref{Jk-ep}) that we may write $M^{-i} v = u + J^{[k]}e_p$ with $u \in V_0$ and $0 \leq k <p$. Then the element $ X^i Y_u Z^k$ of $G$ has $\eta$ as its first component. Since we  already know that $|G| \leq |N|$, this shows both that $G$ acts regularly on $N$, and that the  
expression (\ref{g1}) for each $g \in G$ is unique.
\end{proof} 

We now come to the central result of this paper. 

\begin{theorem} \label{simple}
Under the correspondence described in \S\ref{Hol-cor}, the subgroup $G$ of $\Hol(N)$ defined in (\ref{def-G}) gives a simple skew brace $(B, \cdot, \circ)$ of order $n=p^p q$ with $(B, \cdot) \cong N$ and $(B, \circ) \cong G$.
\end{theorem}
\begin{proof}	
Since $G$ is a regular subgroup of $\Hol(N)$, it corresponds to a skew brace $B$ whose additive and multiplicative groups are as stated. To show that $B$ is simple, recall from Lemma \ref{N-subgroups} that the only nontrivial proper normal subgroup of $N$ is its (unique)  Sylow $p$-subgroup. Since an ideal in $B$ must be normal in both $(B,\cdot)$ and $(B,\circ)$, it follows that $B$ must be simple unless $G$ also has a normal (and hence unique) Sylow $p$-subgroup. However, it is immediate from Lemma \ref{G-rels} and Lemma \ref{G-reg}   
that 
$$ P = \{ Y_v Z^k : v \in V_0, \; 0 \leq k \leq p-1\} $$
is a subgroup of $G$ of order $p^p$, and it is not normal since $X^{-1}ZX=X^{p-1}Z$.
\end{proof}

\begin{remark}
The Sylow $p$-subgroup $P$ of $G$ has exponent $p^2$, and when $p>2$ it has derived length $2$ since its commutator subgroup is the elementary abelian group $\langle Y_{e_{p-2}}, Y_{e_{p-3}},  \ldots , Y_{e_1}\rangle$ of index $p^2$. On the other hand, it follows from Corollary \ref{commutator} that the nilpotency class of  $P$ is $p-1$, the maximum possible for a group of order $p^p$. Thus the structure of $P$ is in some sense extremal. Of course $P$ is abelian if and only if $p=2$. 
\end{remark}
 
\subsection{The opposite skew brace to $B$} \label{op-sect}

The skew brace $B$ constructed in Theorem \ref{simple} has an opposite skew brace $B^\opp$ and, as explained in Lemma \ref{op-lemma}, $B^\opp$ corresponds to a regular subgroup $G^*$ in $\Hol(N)$. Since $G^* \cong G$, the proof of Theorem \ref{simple} also shows that $B^\opp$ is a simple skew brace. For the sake of completeness, we now give explicit generators and relations for $G^*$, analogous to those already given for the regular subgroup $G$ corresponding to $B$. These will not be required in the rest of the paper. 

It is convenient to take the inverses of the generators of $G$ given above: we have
$$ G = \langle X^{-1}, Y_v^{-1}, Z^{-1} : v \in V_0\rangle  $$ 
where we calculate
$$ X^{-1} = \left[  \left( \begin{array}{c|c} M^{-1}  & 0 \\ \hline 0 & 1 \end{array} \right) , \conj \left( \begin{array}{c|c} I &  0 \\ \hline  0 & 1 \end{array} \right) \right], $$
$$ Y_u^{-1} = \left[ \left( \begin{array}{c|c} I & -v \\ \hline 0 & 1 \end{array} \right) , 
\conj \left( \begin{array}{c|c} I &  u \\ \hline  0 & 1 \end{array} \right) \right] \mbox{ for each } u \in V_0, $$ 
$$ Z^{-1} = \left[ \left( \begin{array}{c|c} I & -J^{-1} e_p \\ \hline 0 & 1 \end{array} \right) , 
\conj \left( \begin{array}{c|c} J^{-1} &  J^{-1} e_p \\ \hline  0 & 1 \end{array} \right) \right]. $$
Replacing each of the original generators $g$ of $G$ by the element $(g^{-1})^*$, where the operation $*$ is as in Lemma \ref{op-lemma}, we obtain the following set of generators for the group $G^*$. 

$$ X^* = \left[  \left( \begin{array}{c|c} M  & 0 \\ \hline 0 & 1 \end{array} \right) , \conj \left( \begin{array}{c|c} M^{-1} &  0 \\ \hline  0 & 1 \end{array} \right) \right], $$
$$ Y_u^* = \left[ \left( \begin{array}{c|c} I & u \\ \hline 0 & 1 \end{array} \right) , 
\conj \left( \begin{array}{c|c} I &  0 \\ \hline  0 & 1 \end{array} \right) \right] \mbox{ for each } u \in V_0, $$ 
$$ Z^* = \left[ \left( \begin{array}{c|c} I & J^{-1} e_p \\ \hline 0 & 1 \end{array} \right) , 
\conj \left( \begin{array}{c|c} J^{-1} &  0 \\ \hline  0 & 1 \end{array} \right) \right]. $$

Analogously to Lemma \ref{G-rels}, we find that the elements $X^*$, $Y_v^*$, $Z^*$ satisfy the following relations:
$$ {X^*}^q=1, \qquad  {Z^*}^p=Y^*_{e_1}, \qquad X^* Z^* =Z^* {X^*}^p, $$
$$   X^* Y^*_v  = Y^*_v X^*, \qquad  Y^*_v Y^*_w = Y^*_{v+w}, \qquad 
	Y^*_v Z^* = Z^* Y^*_{Jv}.  $$

\section{$B$ and $B^\opp$ are not isomorphic} \label{non-isom}

We have obtained two mutually opposite simple skew braces $B$, $B^\opp$ of order $n=p^p q$ corresponding to the regular subgroups $G$, $G^*$ in $\Hol(N)$ given in (\ref{def-G}) and \S\ref{op-sect} respectively. Our goal in this section is to verify that $B$ and $B^\opp$ are not isomorphic, so that we have indeed found two distinct simple skew braces of order $n$. Rather than working directly with the generators of $G^*$ given in \S\ref{op-sect}, we adopt a more indirect approach.

If $\Phi: B \to B^\opp$ is an isomorphism of skew braces, then 
$\Phi$ is simultaneously an automorphism of $(B,\circ)$ and an anti-automorphism of $(B,\cdot)$.
The anti-automorphisms of $(B,\cdot)$ are precisely the functions $b \mapsto \alpha(b)^{-1}$ for $\alpha \in \Aut(B,\cdot)$. We therefore look for an automorphism $\alpha$ of $(B,\cdot)$ such that the function 
$\Phi:B \to B$, $\Phi(b)=\alpha(b)^{-1}$, is an automorphism of $(B,\circ)$. By (\ref{def-circ}), this condition on $\Phi$ is equivalent to
\begin{equation} \label{circ-autom}
	g_{\eta} g_{\eta'} = g_\mu \Rightarrow g_{\Phi(\eta)} g_{\Phi(\eta')} = g_{\Phi(\mu)} \mbox{ for } \eta, \eta', \mu \in N,
\end{equation}
where, as before, $g_\eta$ denotes the unique element of $G$ with first component $\eta$.
The bijection $g_\eta \mapsto \eta$ from $G$ to $N$ is described explicitly by (\ref{XYZ}) in the proof of Lemma \ref{G-reg}, and its inverse is described implicitly in that proof. To make the latter description explicit, we introduce the following notation.

\begin{definition} \label{kap-pi}
Define $\kappa: V \to \F_p$ by
$$ \kappa( k_1 e_1 + \cdots + k_p e_p) = k_p \mbox{ for } k_1, \ldots, k_p \in \F_p. $$
For $v \in V$, let $\tkap(v)$ to be the unique integer with $0 \leq \tkap(v) \leq p-1$ whose residue class in $\F_p=\Z/p\Z$ is $\kappa(v)$. 

We then have $v-\tkap(v) e_p = v-\kappa(v) e_p \in V_0$, so also $v-J^{[\tkap(v)]} e_p \in V_0$ by (\ref{Jk-ep}). We further define
$\pi: V \to V_0$ by 
$$ \pi(v) =  v-J^{[\tkap(v)]} e_p. $$  
\end{definition}
\begin{remark}
We note that $\kappa$ is a linear map, but $\pi$ is linear only when $p=2$. The distinction between $\kappa$ and $\tkap$ matters since $J^{[p]} \neq J^{[0]}$, and also since $Z$ has order $p^2$. 
\end{remark}

Now by (\ref{XYZ}), if $i \in \Z$, $u \in V_0$ and $0 \leq k <p$, we have 
\begin{equation} \label{G-to-N}
	g_\eta = X^i Y_u Z^k \Rightarrow \eta = \left( \begin{array}{c|c} M^i & M^i(u +J^{[k]} e_p) \\ \hline 0 & 1 \end{array} \right), 
\end{equation}
and conversely, for any $i \in \Z$ and $v \in V$, we have 
\begin{equation} \label{N-to-G}
	 \eta =  \left( \begin{array}{c|c} M^i & M^i v\\ \hline 0 & 1 \end{array} \right) 
	 \Rightarrow g_\eta = X^i Y_{\pi(v)} Z^{\tkap(v)}.
\end{equation}	 
	 
By Lemma \ref{Aut-N}, the automorphism $\alpha$ of $(B,\cdot) =N$ is given by conjugation by a matrix of the form 
\begin{equation}\label{alpha-aut}
		  \left( \begin{array}{c|c} A & w \\ \hline 0 & 1 \end{array} \right) 
\end{equation}
where $w \in V$ and $A$ satisfies $AMA^{-1}=M^s$. By Lemma \ref{centr-norm}(vii), we have $s=p^j$ with $0 \leq j \leq p-1$. 

\begin{lemma} \label{g-Phi}
Let $\eta \in N$, and write $g_\eta = X^i Y_u Z^k$ with $i \in \Z$, $u \in V_0$, $0 \leq k<p$. Then 
$$ g_{\alpha(\eta)} = X^{is} Y_{\pi(v)}Z^{\tkap(v)} $$
where $v=(M^{-is}-I)w + A(u+J^{[k]}e_p)$, and
$$ g_{\Phi(\eta)} = g_{\alpha(\eta)^{-1}} = X^{-is} Y_{\pi(y)} Z^{\tkap(y)} $$
with $y= (M^{is}-I)w -AM^i(u+J^{[k]} e_p)$.
\end{lemma}
\begin{proof}
As $\eta$ is given by (\ref{G-to-N}), we calculate that 
\begin{eqnarray*}
	 \alpha(\eta) & = & \left( \begin{array}{c|c} M^{is} & (I-M^{is})w +AM^i (u+J^{[k]}e_p) \\ \hline 0 & 1 \end{array} \right) \\
  & = & \left( \begin{array}{c|c} M^{is} & M^{is}\big( (M^{-is}-I)w +A (u+J^{[k]}e_p)\big) \\ \hline 0 & 1 \end{array} \right) .
\end{eqnarray*}
Here we have used that $AM^i=M^{is}A$. The expression for $g_{\alpha(\eta)}$ then follows using (\ref{N-to-G}). Taking the inverse, we obtain
\begin{eqnarray*}
\Phi(\eta) & = & \left( \begin{array}{c|c} M^{-is} & - (M^{-is}-I)w -A(u+J^{[k]}e_p) \\ \hline 0 & 1 \end{array} \right)  \\
 & = & \left( \begin{array}{c|c} M^{-is} &  M^{-is}\big( (M^{is}-I)w -M^{is}A(u+J^{[k]}e_p) \big) \\ \hline 0 & 1 \end{array} \right),
\end{eqnarray*}   
giving the expression for $g_{\Phi(\eta)}$.
\end{proof}

\begin{theorem}   \label{no-isom}
The simple skew brace $B$ given in Theorem \ref{simple} is not isomorphic to its opposite skew brace $B^\opp$. 
\end{theorem}
\begin{proof}
Suppose for a contradiction that the skew braces $B$ and $B^\opp$ are isomorphic. Then, for some choice of $A$ and $w$ in (\ref{alpha-aut}), the corresponding function $\Phi$ is an automorphism of $(N,\circ)$. Thus (\ref{circ-autom}) holds.

We will make two applications of (\ref{circ-autom}). 
First, take $\eta$, $\eta'$, $\mu$ so that $g_\eta=X^i$, $g_{\eta'}=X$, $g_\mu=X^{i+1}$ for $i \in \Z$. For brevity, we set $\pi(i)=\pi\big( (M^{is}-I)w\big)$ and $\tkap(i)=\tkap\big( (M^{is}-I)w \big)$. 
Then by Lemma \ref{g-Phi} we have $g_{\Phi(\eta)} = X^{-is} Y_{\pi(i)} Z^{\tkap(i)}$, with similar expressions for $g_{\Phi(\eta')}$ and $g_{\Phi(\mu)}$. Hence by (\ref{circ-autom}), we have 
$$ X^{-is} Y_{\pi(i)} Z^{\tkap(i)}  X^{-s} Y_{\pi(1)} Z^{\tkap(1)}  =  X^{-(i+1)s} Y_{\pi(i+1)} Z^{\tkap(i+1)}. $$ 
Using the relations $ZX=X^pZ$, $ZY_u = Y_{Ju} Z$ and $Y_u X =X Y_u$ in Lemma \ref{G-rels}, 
and writing $k=\tkap(i)$, this becomes
$$ X^{-is -sp^k} Y_{\pi(i)+J^k\pi(1)} Z^{k+\tkap(1)} =   X^{-(i+1)s} Y_{\pi(i+1)} Z^{\tkap(i+1)}. $$ 
Comparing exponents on $X$, we find that
$$ -is-s p^k \equiv  -(i+1)s \pmod{q}, $$
so $k=\tkap(i)=0$. 
This means that $(M^{is}-I)w \in V_0$ for all $i \in \Z$. By Lemma \ref{centr-norm}(iii) (applied to $M^s$ in place of $M$), we therefore have $w=0$. Thus the formula for $g_{\Phi(\eta)}$ in Lemma \ref{g-Phi} simplifies to 
$$ g_\eta=X^i Y_u Z^k \Rightarrow g_{\Phi(\eta)} = X^{-is} Y_{\pi(-AM^i (u+J^{[k]}e_p))} Z^{\tkap(-A M^i (u+J^{[k]}e_p))}. $$

For our second application of (\ref{circ-autom}), let $g_\eta=Y_u$, $g_\eta'=X$, $g_\mu=Y_u X = X Y_u$ for arbitrary $u \in V_0$. We have 
$$ g_{\Phi(\eta)} = Y_{\pi(-Au)} Z^{\tkap(-Au)}, $$
$$  g_{\Phi(\eta')} = X^{-s}, $$
$$ g_{\Phi(\mu)} = X^{-s} Y_{\pi(-AMu)} Z^{\tkap(-AMu)}. $$
Hence we have the equation
$$  Y_{\pi(-Au)} Z^{\tkap(-Au)}  X^{-s}  = 
   X^{-s} Y_{\pi(-AMu)} Z^{\tkap(-AMu)}  $$
in $G$. Writing $h=\tkap(-Au)$, we therefore have
$$ X^{-sp^h} Y_{\pi(-Au)} Z^h =   X^{-s} Y_{\pi(-AMu)} Z^{\tkap(-AMu)}. $$
Comparing exponents on $X$, we see that $h=0$, and comparing exponents on $Z$, we also find that $\tkap(-AMu)=0$. Hence $-Au \in V_0$ and $-AMu\in V_0$ for all $u \in V_0$. It follows that $M^s V_0 = (-AM)(-A) V_0=V_0$. But by Lemma \ref{centr-norm}(i), there are no $M^s$-invariant subspaces of $V$ except $0$ and $V$, so we have obtained the required contradiction.
\end{proof}

\section{Automorphisms}

It is of interest to determine the group $\Autsb(B)$ of automorphisms of the simple skew braces we have constructed. Since we can do this by a similar argument to Theorem \ref{no-isom}, we include the result here. As observed in \S\ref{opposite}, we automatically have $\Autsb(B^\opp) \cong \Autsb(B)$.

\begin{theorem}  \label{automs}
Let $B$ be the simple skew brace in Theorem \ref{simple}. Then $\Autsb(B)$ has order $p$ and is generated by the automorphism of $(B,\cdot)=N$ given by conjugation by the block matrix 
$$ \left( \begin{array}{c|c} J & 0 \\ \hline 0 & 1 \end{array} \right). $$
\end{theorem}
\begin{proof}
Any automorphism $\alpha$ of $(N,\cdot)$ is induced by conjugation by a matrix as in (\ref{alpha-aut}) for some $w \in V$ and some $A$ with $AMA^{-1}=M^s$. Then $\alpha$ will be an automorphism of the skew brace $B$ if and only if also
\begin{equation} \label{circ-autom2}
	g_{\eta} g_{\eta'} = g_\mu \Rightarrow g_{\alpha(\eta)} g_{\alpha(\eta')} = g_{\alpha(\mu)}.
\end{equation} 	
We apply this for several choices of $\eta$, $\eta'$, $\mu$. 

Firstly, let $g_\eta=X^i$, $g_{\eta'}=X$, $g_\mu=X^{i+1}$ for $i \in \Z$. Writing $\pi(i)=\pi\big( (M^{-is}-I)w\big)$ and $\tkap(i)=\tkap\big( (M^{-is}-I)w \big)$, we see from Lemma \ref{g-Phi} that
$$ g_{\alpha(\eta)} = X^{is} Y_{\pi(i)} Z^{\tkap(i)}, $$
and similarly for $g_{\alpha(\eta')}$ and $g_{\alpha(\mu)}$. This gives the equation
$$ X^{is} Y_{\pi(i)} Z^{\tkap(i)} X^s Y_{\pi(1)} Z^{\tkap(1)} 
        = X^{(i+1)s} Y_{\pi(i+1)} Z^{\tkap(i+1)}. $$
Rewriting the left-hand side and comparing exponents on $X$, we find that $\tkap(i)=0$. Thus $(M^{-is}-I)w=0$  for all $i \in \Z$. By Lemma \ref{centr-norm}(iii) we then have $w=0$.        

Next, take $g_\eta=Y_u Z^k$, $g_{\eta'}=X$, $g_\mu=Y_u Z^k X = X^{p^k} Y_u Z^k$ with $u \in V_0$ and $0 \leq k <p$. So we have
$$ g_{\alpha(\eta)} = Y_{\pi\big( A(u+J^{[k]}e_p) \big)} Z^{\tkap\big( A(u+J^{[k]}e_p) \big)}, $$
$$  g_{\alpha(\eta')} = X^s, $$
$$ g_{\alpha(\mu)} = X^{sp^k} Y_{\pi\big(A(u+J^{[k]}e_p)\big)} Z^{\tkap\big(A(u+J^{[k]}e_p)\big)}. $$
Hence
$$ Y_{\pi\big(A(u+J^{[k]}e_p)\big)} Z^{\tkap\big(A(u+J^{[k]}e_p)\big)} X^s = X^{sp^k} Y_{\pi\big(A(u+J^{[k]}e_p)\big)} Z_{\tkap\big(A(u+J^{[k]}e_p)\big)}. $$
Rewriting and equating exponents on $X$, we find that $\tkap\big(A(u+J^{[k]}e_p)\big)=k$ for all $u \in V_0$ and $0 \leq k <p$. In particular, $AV_0 = V_0$, and $(A-I)e_p \in V_0$. By Lemma \ref{centr-norm}(vii), we have $s=p^j$ for some $j$. Let  $B=AJ^{-j}$. Then $BM=MB$. Since $J V_0=V_0$, we also have $B V_0=V_0$. Hence by Lemma \ref{centr-norm}(v), $B=\lambda I$ for some $\lambda \in \F_p^\times$. Since $(A-I)e_p \in V_0$ and $(J-I)e_p=e_{p-1} \in V_0$, we also have $(B-I)e_p \in V_0$, so $\lambda=1$. We conclude that $A=J^j$. Thus any skew brace automorphism of $B$ arises from conjugation by some power of the block matrix in the statement of the Theorem.  

It remains to verify that conjugation by this block matrix does indeed give a skew brace automorphism. Clearly it will then have order $p$. Thus we take $\alpha$ to be given by $w=0$ and $A=J$. Then $s=p$, and now 
$$ g_\eta = X^i Y_u Z^k \Rightarrow g_{\alpha(\eta)} = X^{ip} Y_{\pi\big(J(u+J^{[k]}e_p)\big)} Z^k $$
where
$$ \pi\big(J(u+J^{[k]}e_p)\big) = Ju + JJ^{[k]} e_p - J^{[k]} e_p = Ju + (J^k-I) e_p. $$
We will check that (\ref{circ-autom2}) holds for arbitrary elements $g_\eta=X^i Y_u Z^k$ and $g_{\eta'}=X^{i'} Y_{u'} Z^{k'}$ of $G$, where $u$, $u' \in V_0$ and $0 \leq k, k' <p$. Let $\mu \in N$ be the element satisfying 
\begin{equation} \label{g-eta}
 g_\mu = g_\eta g_{\eta'} = X^{i+i'p^k} Y_{u+J^ku'} Z^{k+k'}.
\end{equation}
Then we have
\begin{eqnarray}
 	g_{\alpha(\eta)} g_{\alpha(\eta')} & = & X^{ip} Y_{Ju+(J^k-I)e_p} Z^ k X^{i'p} Y_{Ju'+(J^{k'}-I)e_p} Z^{k'} 
 	                                                             \nonumber \\
 	                                 & = & X^{ip+i'p^{k+1}} Y_{Ju + (J^k-I) e_p + J^k\big(Ju'+(J^{k'}-I)e_p\big)} Z^{k+k'}
 	                                            \label{g-alpha-eta}.
\end{eqnarray}  
The expressions (\ref{g-eta}) and (\ref{g-alpha-eta}) are not necessarily in our standard form (\ref{g1}) since possibly $k+k' \geq p$. We therefore distinguish two cases. 

Firstly, if $k+k'<p$, then from (\ref{g-eta}) we have
$$ g_{\alpha(\mu)} = X^{ip+i'p^{k+1}} Y_{J(u+J^ku') +(J^{k+k'}-I)e_p} Z^{k+k'}. $$
The exponents on $X$ and $Z$ agree with those in (\ref{g-alpha-eta}), so (\ref{circ-autom2}) holds in this case if the subscripts on $Y$ coincide, that is, if 
$$  J(u+J^ku') +(J^{k+k'}-I)e_p = Ju + (J^k-I) e_p + J^k \big( Ju'+(J^{k'}-I \big) e_p). $$
Expanding both sides, we see that this identity does indeed hold. 
 
Secondly, if $k+k' \geq p$ then we use the the relation $Z^p=Y_{e_1}$ to rewrite (\ref{g-eta}) and (\ref{g-alpha-eta}) as
$$	g_\mu = g_\eta g_{\eta'} = X^{i+i'p^k} Y_{u+J^ku'+e_1} Z^{k+k'-p},  $$
$$  g_{\alpha(\eta)} g_{\alpha(\eta')} = X^{ip+i'p^{k+1}} Y_{Ju + (J^k-I) e_p + J^k \big( Ju'+(J^{k'}-I)e_p \big) +e_1} Z^{k+k'-p}. $$  
Then 
$$ g_{\alpha(\mu)} = X^{ip+i'p^{k+1}} Y_{J(u+J^ku'+e_1)+(J^{k+k'-p}-I)e_p} Z^{k+k'-p}, $$
which agrees with $g_{\alpha(\eta)} g_{\alpha(\eta')}$ provided that
\begin{eqnarray*}
	\lefteqn{J(u+J^ku'+e_1)+(J^{k+k'-p}-I)e_p = } \\
	& Ju + (J^k-I) e_p + J^k \big( Ju'+(J^{k'}-I) e_p \big) +e_1. 
\end{eqnarray*}
Expanding both sides, and recalling that $Je_1=e_1$ and $J^p=I$, we again find that the required identity holds.
\end{proof}

\section{Classification of simple skew braces order $n$}

Our goal in this final section is the following result.

\begin{theorem}  \label{uniqueness}
Any simple skew brace of order n as in Notation \ref{pq} is isomorphic to either the skew brace $B$ of Theorem \ref{simple} or its opposite skew brace $B^\opp$.
\end{theorem}

\subsection{Determining the additive and multiplicative groups}

\begin{lemma}  \label{ppq-grps} 
	Let $H$ be a group of order $p^p q$, and let $P$ and $Q$ be a Sylow $p$-subgroup and a Sylow $q$-subgroup respectively. Then at least one of $P$ and $Q$ is normal in $H$, and $H$ has one of the following structures:
\begin{itemize}
		\item[(i)] $H=P \times Q$;
		\item[(ii)] $H \cong N$ with $N$ as in Lemma $\ref{constr-N}$, so that in particular $H$ is a semidirect product $P \rtimes Q$ and $Q$ is not normal in $H$;
		\item[(iii)] $H$ is a semidirect product $Q \rtimes P$ and $P$ is not normal in $H$.
\end{itemize}
\end{lemma}
\begin{proof}
If $P \lhd H$ then either $H$ is a direct product as in (i), or $Q$ is not normal in $H$. In the latter case, Lemma \ref{constr-N} shows that (ii) holds. It therefore suffices to show that if $P$ is not normal in $H$ then $Q \lhd H$, since in that case $P$ is a complement to $Q$ and (iii) holds. 

Suppose that $P$ is not normal in $H$, and let
$$ \{1\} = H_0 \lhd H_1 \lhd \ldots \lhd H_{p+1} = H $$
be a composition series for $H$. Then exactly one quotient $H_{s+1}/H_s$ has order $q$, and $s<p$ since otherwise $H_p$ would be a normal Sylow $p$-subgroup of $H$. Writing $P_i=H_i \cap P$ for each $i$, we obtain the series
$$ \{1\} = P_0 \lhd P_1 \lhd \ldots P_s = P_{s+1}  \lhd  \cdots \lhd  P_{p+1} = P, $$
which is a composition series of $P$ with one term repeated. Then $H_s=P_s$ and $H_{s+1} = P_s Q$ for some choice of Sylow $q$-subgroup $Q$ of $H$. Thus $P_s \lhd P_s Q$, and $Q$ acts by conjugation on $P_s$. As $P_s$ has order $p^s$ with $s<p$, and the order of $p$ mod $q$ is $p$, a Frattini argument (as in the proof of Lemma \ref{constr-N}) shows that $P_s$ has no automorphism of order $q$. Thus $Q$ acts trivially on $P_s$ and $Q \lhd P_s Q=H_{s+1}$. Now if $s+1 \leq t  \leq p$ and $Q \lhd H_t$ then $Q$ is the unique Sylow $q$-subgroup of $H_t$ and is therefore characteristic in $H_t$. As $H_t \lhd H_{t+1}$, it follows that $Q \lhd H_{t+1}$. By induction, we therefore have $Q \lhd H_{p+1}=H$ as required.
\end{proof}

\begin{lemma}  \label{N-G}
Let $H$ be a group of order $p^p q$ and let $G$ be a regular subgroup of $\Hol(H)$ giving a skew brace $B$. If $B$ is simple then, in the classification of Lemma \ref{ppq-grps},  $H$ is of type (ii) and $G$ is of type (iii).  
\end{lemma}
\begin{proof}
If $H$ contains a normal Sylow $p$-subgroup $P$ then $P$ is characteristic in $H$ and, by Remark \ref{char}, gives a left ideal in $B$.  This is then a Sylow $p$-subgroup of $G$. If this subgroup is normal then $B$ contains an ideal of order $p^p$. Similarly, if both $H$ and $G$ contain a normal Sylow $q$-subgroup, then $B$ contains an ideal of order $q$. Thus, for $B$ to be simple, either $H$ is of type (ii) and $G$ of type (iii), or vice versa. Suppose for a contradiction that the second of these possibilities holds: $H = Q \rtimes P$ and $G=V \rtimes C_q \cong N$. 
Since $q<p^p$, the kernel of the action of $P$ on $Q$ is a subgroup $P_0$ of $P$ of order $p^i$ where $1 \leq i \leq p-1$. Then $Q \times P_0$ is a subgroup of $H$ of order $p^i q$, contradicting Lemma \ref{N-subgroups}.
\end{proof}

\subsection{Regular subgroups}

By Lemma \ref{N-G}, any simple skew brace of order $n$ comes from a regular subgroup in $\Hol(N)$, where $N$ is the group constructed in Lemma \ref{constr-N}. Moreover, we only need to consider regular subgroups with the structure $Q \rtimes P$ as in Lemma \ref{ppq-grps}(iii). Our aim is to show that, up to conjugation by $\Aut(N)$, the only possibilities are the group $G$ in (\ref{def-G}) and the corresponding group $G^*$ obtained from $G$ via Lemma \ref{op-lemma}. 

Any such subgroup contains a normal subgroup of order $q$. Let us begin by looking for suitable subgroups of order $q$ in $\Hol(N)$. We will say that a subgroup of $\Hol(N)$ acts freely on $N$ if the stabiliser of each element of $N$ is trivial. Any subgroup of a regular group acts freely. 

\begin{lemma} \label{q-semireg}
	Up to conjugation by $\Aut(N)$, any subgroup of order $q$ in $\Hol(N)$ which acts freely on $N$ has a generator of one of the following forms:
	\begin{itemize}
		\item[(i)]
		$$ X = \left[  \left( \begin{array}{c|c} M  & 0 \\ \hline 0 & 1 \end{array}
		\right) , \conj \left( \begin{array}{c|c} I &  0 \\ \hline  0 & 1 \end{array} \right) \right]; $$
		\item[(ii)] 
		$$ X = \left[  \left( \begin{array}{c|c} M^{-1} & 0 \\ \hline 0 & 1 \end{array} \right) , \conj \left( \begin{array}{c|c} M &  0 \\ \hline  0 & 1 \end{array} \right) \right]; $$
		\item[(iii)] 
		$$ X = \left[  \left( \begin{array}{c|c} M^k & v \\ \hline 0 & 1 \end{array} \right) , \conj \left( \begin{array}{c|c} M &  0 \\ \hline  0 & 1 \end{array} \right) \right] $$ 
		with $1 \leq k \leq q-2$ and $v$ arbitrary. 
	\end{itemize}
Moreover the generator in (ii) is obtained from that in (i) via the operation $*$ in Lemma \ref{op-lemma}.
\end{lemma}
\begin{proof}
A typical element $X$ of $\Hol(N)$ has the form 
$$ X = \left[  \left( \begin{array}{c|c} M^k  & v \\ \hline 0 & 1 \end{array} \right) , \conj \left( \begin{array}{c|c} A &  w \\ \hline  0 & 1 \end{array} \right) \right], $$ 
where $A$ normalises $\langle M \rangle$, say $AMA^{-1}=M^s$. 

We determine when $X$ has order $q$. The second component of $X^q$ is 
$$   \conj \left( \begin{array}{c|c} A^q &  A^{[q]} w \\ \hline  0 & 1 \end{array} \right) . $$ 
Thus we need $A^q=I$ and $A^{[q]}w=0$. There are two possibilities:
\begin{itemize}
\item[(a)] $A=I$, so $A^{[q]}=qI$ with $w=0$;
\item[(b)] $A$ has order $q$, so $A=M^r$ with $1 \leq r \leq q-1$ by Lemma \ref{centr-norm}(vii). Then $A-I$ is invertible and $A^{[q]} = (A-I)^{-1} (A^q-I)=0$ with $w$ arbitrary.
\end{itemize}

In case (a) we have 
$$ X = \left[  \left( \begin{array}{c|c} M^k  & v \\ \hline 0 & 1 \end{array} \right) , \conj \left( \begin{array}{c|c} I &  0 \\ \hline  0 & 1 \end{array} \right) \right]. $$ 
For $X$ to have order $q$, we also need $k \neq 0$. Replacing $X$ by a suitable power, we may assume $k=1$.
Let $u=(M-I)^{-1}v$. Inside $\Hol(N)$, we conjugate $X$ by the automorphism 
\begin{equation} \label{conj-aut}
	 \left( \begin{array}{c|c} I &  u \\ \hline  0 & 1 \end{array} \right) 
\end{equation}
of $N$, i.e.~by the element
$$	\left[  \left( \begin{array}{c|c} I  & 0 \\ \hline 0 & 1 \end{array} \right) , \conj \left( \begin{array}{c|c} I &  u \\ \hline  0 & 1 \end{array} \right) \right]  $$
of $\Hol(N)$. This replaces $v$ by $0$, so we obtain the element in (i).

In case (b), $X$ has the form 
$$ X = \left[  \left( \begin{array}{c|c} M^k  & v \\ \hline 0 & 1 \end{array} \right) , \conj \left( \begin{array}{c|c} M^r &  w \\ \hline  0 & 1 \end{array} \right) \right]. $$ 
for some $v$, $w$. Replacing $X$ by a suitable power, we may assume that $r=1$.

If $k=0$ then 
$$ X = \left[  \left( \begin{array}{c|c} I  & v \\ \hline 0 & 1 \end{array} \right) , \conj \left( \begin{array}{c|c} M &  w \\ \hline  0 & 1 \end{array} \right) \right] $$ 
which fixes
$$ \left( \begin{array}{c|c}I &  (I-M^{-1}) v \\ \hline  0 & 1 \end{array} \right),  $$
so the subgroup generated by $X$ does not act freely.

If $k \neq 0$, we may conjugate by the automorphism of the form (\ref{conj-aut}) with $u=(M-I)^{-1}w$. Thus, replacing $v$ by some other element of $V$, we can assume that
$$ X = \left[  \left( \begin{array}{c|c} M^k & v \\ \hline 0 & 1 \end{array} \right) , \conj \left( \begin{array}{c|c} M &  0 \\ \hline  0 & 1 \end{array} \right) \right]. $$ 
Inductively, we then have
$$ X^j = \left[  \left( \begin{array}{c|c} M^{kj} & (M^{k+1})^{[j]}v \\ \hline 0 & 1 \end{array} \right) , \conj \left( \begin{array}{c|c} M^j &  0 \\ \hline  0 & 1 \end{array} \right) \right]  \mbox{for } j \geq 1. $$  
If $k=q-1$, we have $(M^{k+1})^{[q]}v =qv$ so $v=0$. Then
$$ X = \left[  \left( \begin{array}{c|c} M^{-1} & 0 \\ \hline 0 & 1 \end{array} \right) , \conj \left( \begin{array}{c|c} M &  0 \\ \hline  0 & 1 \end{array} \right) \right] $$
is as in (ii). If $1 \leq k \leq q-2$ then $(M^{k+1})^{[q]}=0$ and we obtain no restriction on $v$, giving (iii).  

Finally, applying the operation $*$ from Lemma \ref{op-lemma} to the generator $X$ in (i), we obtain
$$ \left[  \left( \begin{array}{c|c} M  & 0 \\ \hline 0 & 1 \end{array}
\right)^{-1} , \conj \left( \begin{array}{c|c} M  & 0 \\ \hline 0 & 1 \end{array}
\right) \conj \left( \begin{array}{c|c} I &  0 \\ \hline  0 & 1 \end{array} \right) \right],$$
which coincides with the generator in (ii).
\end{proof}

We will look for regular subgroups of $\Hol(N)$ containing a normal subgroup of order $q$ generated by an element $X$ of order $q$ of type (i) or (iii) in Lemma \ref{q-semireg}. We do not need to consider type (ii) since the resulting groups will be obtained by applying the operation $*$ of Lemma \ref{op-lemma} to those coming from type (i).

\begin{lemma} \label{type-i}
For $X$ as in Lemma \ref{q-semireg}(i), the corresponding skew brace is the simple skew brace $B$ of Theorem \ref{simple}.
\end{lemma}
\begin{proof}
We take 
$$ X = \left[  \left( \begin{array}{c|c} M  & 0 \\ \hline 0 & 1 \end{array}
\right) , \conj \left( \begin{array}{c|c} I &  0 \\ \hline  0 & 1 \end{array} \right) \right] $$
and look for elements of $p$-power order normalising $\langle X \rangle$. 
If
$$ U = \left[  \left( \begin{array}{c|c} M^k  & v \\ \hline 0 & 1 \end{array} \right) , \conj \left( \begin{array}{c|c} A &  w \\ \hline  0 & 1 \end{array} \right) \right] $$ 
is such an element, with $UXU^{-1}=X^s$, then expanding $UX=X^s U$ we find
$$ AMA^{-1} = M^s, \qquad v +M^k w - M^k A M A^{-1} w = M^s v. $$
We conclude firstly by Lemma \ref{centr-norm}(vii) that $s=p^j$ for some $j$, so that conjugation by $U$ gives an automorphism of $\langle X \rangle$ of order $1$ or $p$, and secondly that 
$$ (I-M^s)v = - (I - M^s) M^k w. $$
Since $I-M^s$ is invertible by Lemma \ref{centr-norm}(ii), we have 
$$ v= - M^k w. $$

Recall that are looking for a regular subgroup of $\Hol(N)$ with the structure $Q \rtimes P$, where $Q=\langle X \rangle$ and the group $P$ of order $p^p$ acts nontrivially on $Q$. Thus the kernel $P_0$ of the action of $P$ on $Q$ has order $p^{p-1}$ and consists of elements $U \in \Hol(N)$ as just described which commute with $X$. The group $P$ itself will be generated by $P_0$ and one further element $Z$ with $ZXZ^{-1}=X^p$. 

If $U \in P_0$ then $A$ commutes with $M$ and has $p$-power order. By Lemma \ref{centr-norm}(iv), we then have $A=I$. For $U$ to have order a power of $p$, we then require $k=0$. Hence $U$ is of the form
$$ Y_v = \left[  \left( \begin{array}{c|c} I  & v \\ \hline 0 & 1 \end{array} \right) , \conj \left( \begin{array}{c|c} I &  -v \\ \hline  0 & 1 \end{array} \right) \right] $$  
for some $v \in V$. For elements of this form, we have $Y_v Y_w = Y_{v+w}$. Let 
$$ V' = \{ u \in V: Y_u \in P_0 \}. $$
Then $V'$ is a $(p-1)$-dimensional $\F_p$-subspace  of $V$, and the function $u \mapsto Y_u$ is an isomorphism between $V'$ and $P_0$ as abelian groups. In particular, $P_0$ is an elementary abelian group of order $p^{p-1}$.

For the remaining generator $Z$, we require that $ZX=X^p Z$, $Z^p \in P_0$, and $Z$ normalises $P_0$. Multiplying $Z$ by a power of $X$, we may assume that
$$ Z = \left[  \left( \begin{array}{c|c} I  & v \\ \hline 0 & 1 \end{array} \right) , \conj \left( \begin{array}{c|c} A &  -v \\ \hline  0 & 1 \end{array} \right) \right] $$ 
for some $v \in V$, where $AMA^{-1}=M^p$. Then $v \not \in V'$ since our group must be transitive on $N$. Also, $A$ must have order $p$, so that $A=T^{(p-1)h} J$ for some $h$ by Lemma \ref{centr-norm}(vii). Since $JTJ^{-1}=T^p$, we have $T^h A T^{-h}=J$, and conjugation by the automorphism
$$ \left( \begin{array}{c|c} T^h &  0 \\ \hline  0 & 1 \end{array} \right) $$  
of $N$ replaces $A$ by $J$, $v$ by $T^h v$, and $V'$ by $T^h V'$. We may therefore assume that $A=J$. We then
have $ZY_u = Y_{Ju} Z$, so $V'$ must be the unique subspace of $V$ of dimension $p-1$ which is invariant under $J$, namely $V_0$ (c.f.~Notation \ref{V0}). Replacing $Z$ by $Y_u Z$ for a suitable $u \in V_0$, we can suppose that $v= \lambda e_p$ for some $\lambda \in \F_p^\times$. After a final conjugation by 
$$	 \left( \begin{array}{c|c} \lambda^{-1} I &  0 \\ \hline  0 & 1 \end{array} \right), $$
we have $v=e_p$. Then $X$, $Y_u$ and $Z$ are the same elements as in (\ref{def-G}) and we obtain the simple skew brace $B$ of Theorem \ref{simple}.
\end{proof}

\begin{lemma}  \label{type-iii}
For $X$ as in Lemma \ref{q-semireg}(iii), there are no regular subgroups $G$ in $\Hol(N)$ containing $X$.
\end{lemma}
\begin{proof}
Let
$$ X = \left[  \left( \begin{array}{c|c} M^k & v \\ \hline 0 & 1 \end{array} \right) , \conj \left( \begin{array}{c|c} M &  0 \\ \hline  0 & 1 \end{array} \right) \right] $$ 
with $1 \leq k \leq q-2$ and $v$ arbitrary. Then
$$ X^s = \left[  \left( \begin{array}{c|c} M^{ks} & (M^{k+1})^{[s]} v \\ \hline 0 & 1 \end{array} \right) , \conj \left( \begin{array}{c|c} M^s &  0 \\ \hline  0 & 1 \end{array} \right) \right]. $$ 
Thus if
$$ U = \left[  \left( \begin{array}{c|c} M^t  & u \\ \hline 0 & 1 \end{array} \right) , \conj \left( \begin{array}{c|c} A &  w \\ \hline  0 & 1 \end{array} \right) \right] $$ 
satisfies $UX=X^s U$ then we have from the second components that 
$$ AM=M^s A , \qquad w = M^s w. $$
By Lemma \ref{centr-norm}(vii), $s$ is a power of $p$, and by Lemma \ref{centr-norm}(ii), $w=0$. Moreover, the automorphism obtained by conjugation by $U$ has order $1$ or $p$, so there must be a subgroup of order $p^{p-1}$ consisting of such elements $U$ for which $s=1$. 

As before, if $s=1$ and $A$ has $p$-power order then $A=I$ and we obtain 
$$ U = \left[  \left( \begin{array}{c|c} M^t  & u \\ \hline 0 & 1 \end{array} \right) , \conj \left( \begin{array}{c|c} I &  0 \\ \hline  0 & 1 \end{array} \right) \right]. $$ 
This element cannot have order a power of $p$ unless $t=0$. Thus there is a subspace $V'$ of $V$ of dimension $p-1$ such that $G$ contains the elements
$$ Y_u = \left[  \left( \begin{array}{c|c} I  & u \\ \hline 0 & 1 \end{array} \right) , \conj \left( \begin{array}{c|c} I &  0 \\ \hline  0 & 1 \end{array} \right) \right], \qquad u \in V'. $$ 

From the first component of the relation $Y_u X=X Y_u$, we then obtain
$$ u + v = v +M^{k+1}u, $$
As $k \neq q-2$, the matrix $M^{k+1}-I$ is invertible, so $u=0$. Thus $V'=0$. This is a contradiction since $V'$ has dimension $p-1$. Hence there is no simple skew brace with $X$ of type (iii) in Lemma \ref{q-semireg}.
\end{proof}

Theorem \ref{uniqueness} now follows immediately from Lemmas \ref{q-semireg}, \ref{type-i} and \ref{type-iii}.

\bibliography{SimpleSkewBracesBib}

\end{document}